\documentclass[10pt]{amsart}

\usepackage{fancyhdr,xypic}
\usepackage{bm}
\usepackage{amssymb}

\newtheorem{thm}{Theorem}[section]
\newtheorem{cor}[thm]{Corollary}
\newtheorem{lem}[thm]{Lemma}
\newtheorem{prop}[thm]{Proposition}
\theoremstyle{definition}	

\newtheorem{rem}[thm]{Remark}

\newtheorem{proof-of-csc}[thm]{Proof}

\def\goth{\mathfrak}

\def\a{\alpha}
\def\b{\beta}

\def\G{\Gamma}

\def\Oa{\mathcal{O}}

\def\FF{\mathbb{F}}
\def\GG{\mathbb{G}}

\def\QQ{\mathbb{Q}}

\def\W{\mathbb{W}}
\def\Z{\mathbb{Z}}

\def\Ext{\mathrm{Ext}}
\def\Gal{\mathrm{Gal}}
\def\Hom{\mathrm{Hom}}

\def\holim{\mathrm{holim}}
\def\lim{\mathrm{lim}}

\def\defeq{\overset{\mathrm{def}}=}
\def\ZZ{{{\Z}}}
\def\longr{{\longrightarrow}\ }

\def\Et{{{(E_2)}}}

\newcommand{\F}{\mathbb{F}}
\newcommand{\s}{\mathbb{S}}

\date{\today} 
\begin{document}

\title{The rational homotopy of the $K(2)$-local sphere 
and the chromatic splitting conjecture for the prime $3$ and 
level $2$} 

\begin{abstract}
We calculate the rational homotopy of the $K(2)$-local sphere 
$L_{K(2)}S^0$ at the prime $3$ and confirm Hopkins' chromatic splitting 
conjecture for $p=3$ and $n=2$.     
\end{abstract}
 
\author{Paul G. Goerss,
Hans-Werner Henn, and Mark Mahowald
}
\thanks{The first author was partially supported by the National Science 
Foundation (USA). The second author was partially supported by ANR ``HGRT''}

\address{Department of Mathematics, 
Northwestern University, Evanston, IL 60208, U.S.A.}  
\address{Institut de Recherche Math\'ematique Avanc\'ee,
C.N.R.S. - Universit\'e de Strasbourg, F-67084 Strasbourg,
France}
\address{Department of Mathematics, 
Northwestern University, Evanston, IL 60208, U.S.A.}  
\maketitle 

\section{Introduction}

Let $K(n)$ be the $n$-th Morava $K$-theory at a fixed prime $p$. 
The Adams-Novikov Spectral Sequence  for computing the
homotopy groups of the $K(n)$-local sphere $L_{K(n)}S^0$
can be identified by \cite{DH} with a descent spectral sequence  
\begin{equation}\label{descentss}
E_2^{s,t}\cong H^s(\GG_n,(E_n)_t)\Longrightarrow \pi_{t-s}(L_{K(n)}S^0) \ . 
\end{equation}
Here $\GG_n$ denotes the automorphism group of the pair
$(\FF_{p^n},\Gamma_n)$, where $\Gamma_n$ is the Honda formal 
group law;  the group $\GG_n$ is a profinite group and cohomology
is continuous cohomology. The spectrum $E_n$ is the  
$2$-periodic Landweber exact  ring spectrum so that the 
complete local ring $(E_n)_0$ classifies deformations
of $\Gamma_n$. 

In this paper we focus on the case $p=3$ and $n=2$. 
In \cite{GHMR}, we constructed  a resolution of the $K(2)$-local sphere
at the prime $3$ using homotopy fixed point spectra of the
form $E_2^{hF}$ where $F \subseteq \GG_2$ is a finite subgroup.
These fixed point spectra are well-understood. In particular, their
homotopy groups have all been calculated (see \cite{GHMR}) and they
are closely related to the Hopkins-Miller spectrum
of topological modular forms. 
The resolution was used in \cite{HKM} to redo and 
refine the earlier calculation of the homotopy of the 
$K(2)$-localization of the mod-$3$ Moore spectrum 
\cite{shi1}. In this paper we show how the results of \cite{HKM} imply 
the calculation of the rational homotopy of the $K(2)$-local sphere. 
Let $\QQ_p$ be the field of fractions of the
$p$-adic integers and $\Lambda$ the exterior algebra functor.
 
\begin{thm}\label{maintheorem} 
There are elements $\zeta\in \pi_{-1}(L_{K(2)}S^0)$ and 
$e\in \pi_{-3}(L_{K(2)}S^0)$ that induce an isomorphism 
of algebras 
$$
\Lambda_{\QQ_3}(\zeta,e) \cong \pi_*(L_{K(2)}S^0)\otimes\QQ\ .
$$
\end{thm} 

Our result is  in agreement with the result predicted by 
Hopkins' chromatic splitting conjecture \cite{hov}, and in fact, 
we will establish this splitting conjecture for $n=2$ and $p=3$. 

We will prove a  more general result which will be useful
for calculations with the Picard group of Hopkins \cite{Picard}.  
Before stating that, let us give some notation.

If $X$ is a spectrum, then we define
$$
(E_n)_\ast X \defeq \pi_\ast L_{K(n)}(E_n \wedge X).
$$ 
Despite the notation, $(E_n)_\ast(-)$ is not quite a homology theory,
because it doesn't take wedges to sums; however, it is a sensitive
and tested algebraic invariant for the $K(n)$-local category.
The ${E_n}_*$-module $(E_n)_*X$ is equipped with the
${\goth m}$-adic topology 
where ${\goth m}$ is the maximal ideal in $(E_n)_0$; this topology 
is always topologically complete but need not be separated. 
With respect to this topology, the group $\GG_n$ acts through
continuous maps and the action is twisted because it
is compatible with the action of $\GG_n$ on the coefficient ring
$(E_n)_\ast$. See \cite{GHMR} \S 2 for
some precise assumptions which guarantee that $(E_n)_*X$ 
is complete and separated. All modules which will be used in this paper 
will in fact satisfy these assumptions. Let $E(n)$ denote the
$n$th Johnson-Wilson spectrum and $L_n$ localization  with respect
to $E(n)$. Note that $E(0)_\ast $ is rational homology and $E(1)$ is
the Adams summand of $p$-local complex $K$-theory.
Let $S_p^n$ denote the $p$-adic completion of the sphere.

\begin{thm}\label{maintheorem2} Let $p=3$ and let $X$ be any $K(2)$-local
spectrum so that $(E_2)_\ast X \cong (E_2)_\ast \cong (E_2)_\ast S^0$
as a twisted $\GG_2$-module. Then there is a weak equivalence
of $E(1)$-local spectra
$$
L_1X \cong L_1(S_3^0 \vee S_3^{-1}) \vee L_0(S_3^{-3}\vee S_3^{-4}).
$$
\end{thm}

We will use Theorem \ref{maintheorem} to prove Theorem
\ref{maintheorem2},  but we note that Theorem \ref{maintheorem}
is subsumed into Theorem \ref{maintheorem2}. Indeed,
$\pi_\ast X\otimes \QQ \cong \pi_\ast L_1X\otimes \QQ$ and
$$
\pi_\ast L_1S_3^0 \otimes \QQ \cong \pi_\ast L_0S_3^0 \cong \QQ_3
$$
all concentrated in degree zero. The generality of the statement
of Theorem \ref{maintheorem2} is not vacuous; there are such
$X$ which are not weakly equivalent to $L_{K(2)}S^0$ -- ``exotic''
elements in the $K(2)$-local Picard group. See \cite{GHMRpic} and
\cite{shiPic}.

We remark that  Theorem \ref{maintheorem} disagrees with 
the calculation by Shimomura and Wang  in
\cite{sh-w}.  In particular, Shimomura and Wang find the
exterior algebra on $\zeta$ only.  

An interesting feature of our proof of Theorem \ref{maintheorem} 
is that it does not require a preliminary calculation of all of
$\pi_*(L_{K(2)}S^0)$.  In fact, we get away with much less, namely 
with only a (partial) understanding of the $E_2$-term of the 
Adams-Novikov Spectral Sequence converging to 
$\pi_*L_{K(2)}(S/3)$ where $S/3$ denotes the mod-$3$ Moore spectrum 
(see Corollary \ref{whats-needed}). Our method of proof 
can also be used to recover the rational homotopy of $L_{K(2)}S^0$ 
as well as the chromatic splitting conjecture 
at primes $p>3$ \cite{sh-y}; we only need to use the analogous corollary 
on the $E_2$-term of the Adams-Novikov Spectral Sequence for the 
$K(2)$-localization of the mod-$p$ Moore spectrum. 

In section 2 we give some general background on the automorphism
group $\GG_2$ and we review the main results of \cite{GHMR}. In section 3 
we recall those results of \cite{HKM} which are relevant for the 
purpose of this paper. Section 4 gives the calculation of the rational 
homotopy groups of $L_{K(2)}S^0$ and in the final section 5 we prove 
Theorem \ref{maintheorem2} and the chromatic splitting conjecture for $n=2$ 
and $p=3$. See Corollary \ref{chrom-split-2}.

\section{Background}

Let  $\G_2$ be the Honda formal group law of height $2$; this is the
unique $3$-typical formal group law over $\FF_9$ with $3$-series
$[3] (x) = x^9$. We begin with a short analysis of the Morava stabilizer
group $\GG_2$, the group of automorphisms of the pair $(\FF_9,\G_2)$.
Let  $\W = \W(\FF_9)$ denote the Witt vectors of $\FF_9$ and
$$
\Oa_2=\W\langle S \rangle /(S^2=3, wS=Sw^\sigma)\ . 
$$ 
Then $\Oa_2$ is isomorphic to the ring of endomorphisms of $\Gamma_2$
over
$\FF_9$; hence $\Oa_2^{\times}$ is isomorphic to the group $\s_2$ of
automorphisms  of $\G_2$ over $\F_9$. Since $\G_2$ is defined over $\F_3$,
there is a splitting 
$$
\GG_2\cong \s_2\rtimes \Gal(\FF_9/\FF_3)
$$ 
with Galois action given by $\phi(x+yS)=x^{\sigma}+y^{\sigma}S$. Here 
$x,y\in \W$ and $(-)^\sigma$ denotes the lift of Frobenius to the Witt vectors.

The $3$-adic analytic group $\s_2 \subseteq \GG_2$  contains elements
of order $3$; indeed, an explicit such element is given by 
$$
a=-\frac{1}{2}(1+\omega S)
$$ 
where $\omega$ is a fixed primitive $8$-th root of unity in 
$\W$. If $C_3$ is the cyclic group of order $3$, the map
$H^\ast(\s_2,\F_3) \to H^\ast(C_3,\F_3)$ defined by $a$ is surjective and,
hence,  $\s_2$ and $\GG_2$ cannot have  finite cohomological
dimension. As a consequence, the trivial module $\Z_3$ 
cannot admit a projective resolution of finite length.  
Nonetheless, $\GG_2$ has virtual finite cohomological dimension,
and admits a finite length resolution by permutation modules obtained from
finite subgroups.  Such a
resolution was constructed in \cite{GHMR} using the following two 
finite subgroups of $\GG_2$. The notation $\langle - \rangle$ indicates
the subgroup generated by the listed elements.
\begin{enumerate}

\item Let  $G_{24}=\langle a,\omega^2,\omega\phi\rangle 
\cong C_3 \rtimes  Q_8$. Here $Q_8$ is the quaternion group of
order $8$. Note $\omega^2$  acts non-trivially and $\omega\phi$
acts trivially on $C_3$. 

\item $SD_{16}=\langle \omega, \phi\rangle $. This subgroup is
isomorphic to the semidihedral group of
order $16$. 
\end{enumerate}

\begin{rem}\label{split-g2} The group $\GG_2$ splits
as a product $\GG_2\cong \GG_2^1\times \Z_3$.
To be specific, the center of $\GG_2$ is isomorphic to $\Z_3^\times$ and
there is an isomorphism from the additive group
$\ZZ_3$ onto the multiplicative subgroup
$1+ 3\Z_3 \subseteq \Z_3^\times$ sending $1$ to $4$. There is
also  a reduced determinant map $\GG_2 \to \Z_3$. (See
\cite{GHMR}.) The
composition $\Z_3 \to \GG_2 \to \Z_3$ is multiplication by
$2$, giving the splitting.  All finite subgroups of $\GG_2$
are automatically finite subgroups of $\GG_2^1$.
\end{rem}

Because
of this splitting, any resolution of the trivial $\GG_2^1$-module
$\Z_3$ can be promoted to a resolution of the trivial
$\GG_2$-module. See Remark \ref{full-grp} below. Thus we begin with $\GG_2^1$.

If $X=\lim X_\alpha $ is a profinite set, define $\Z_3[[X]] =
\lim\ \Z/3^n[X_\alpha]$. The following is the main algebraic
result of \cite{GHMR}.

\begin{thm}\label{alg-res} There is an exact complex of 
$\Z_3[[\GG_2^1]]$-modules 
of the form 
$$
0\to C_3\to C_2\to C_1\to C_0\to \Z_3 
$$ 
with
$$
C_0=C_3\cong \Z_3[[\GG_2^1/G_{24}]]
$$
and 
\begin{align*}
C_1=C_2\cong \Z_3[[\GG_2^1]]\otimes_{\Z_3[SD_{16}]}\Z_3(\chi)\\
\end{align*}
where  $\Z_3(\chi)$ is the $SD_{16}$ module which is free of
rank 1 over $\Z_3$ and with $\omega$ and $\phi$ both acting by 
multiplication by $-1$. 
\end{thm}

We recall that a continuous $\ZZ_3[[\GG_2]]$-module $M$ is 
{\it profinite} if there is an isomorphism 
$M\cong\lim_{\a} M_\alpha$ 
where each $M_\alpha$ is a finite $\ZZ_3[[\GG_2]]$ module.
 
\begin{cor}\label{SS} Let $M$ be a profinite $\Z_3[[\GG_2^1]]$-module. 
Then there is a first quadrant cohomology spectral sequence 
$$
E_1^{p,q}(M)\cong \Ext_{\Z_3[[\GG_2^1]]}^q(C_p,M)\Longrightarrow 
H^{p+q}(\GG_2^1,M)
$$ 
with 
$$
E_1^{0,q}(M)=E_1^{3,q}(M)\cong H^q(G_{24},M)
$$ 
and 
$$
E_1^{1,q}(M)=E_1^{2,q}(M)\cong 
\begin{cases} 
\Hom_{\Z_3[SD_{16}]}(\Z_3(\chi),M) &  q=0  \\ 
0 & q>0 \ .  \\
\end{cases} 
$$ 
\end{cor}

\begin{rem}\label{full-grp}These ideas and techniques can easily be 
extended to the full group $\GG_2$ using the splitting 
$\GG_2 \cong \GG_2^1 \times \Z_3$. Let $\psi \in \ZZ_3$ be a
topological generator; then there is a resolution
$$
\xymatrix{
0 \rto &\ZZ_3[[\ZZ_3]] \rto^-{\psi-1} &\ZZ[[\ZZ_3]] \rto & \ZZ_3 \rto &0\ .
}
$$
Write $C_\bullet \to \ZZ_3$ for the resolution of Theorem \ref{alg-res}.
Then the total complex of the double complex 
$$
C_\bullet \hat{\otimes}
\{ \xymatrix{\ZZ_3[[\ZZ_3]] \rto^-{\psi-1} &\ZZ[[\ZZ_3]] } \}
$$
defines an exact complex $D_\bullet \to \ZZ_3$ of $\ZZ_3[[\GG_2]]$-modules. The symbol
$\hat{\otimes}$ indicates the completion of the tensor product. From
this we get a spectral sequence analogous to that of Corollary  \ref{SS}.
\end{rem}

\begin{rem}\label{profinite} In our arguments below, we will use
the functors on profinite $\Z_3[[\GG_2^1]]$-modules 
to profinite abelian groups given by 
$$
M\mapsto E_2^{p,0}(M) =H^p(\Hom_{\Z_3[[\GG_2^1]]}(C_{\bullet},M))\ . 
$$ 
Here $C_\bullet$ is the resolution of Theorem \ref{alg-res}; thus, we 
are using the $q=0$ line of the $E_2$-page of the spectral
sequence of Corollary \ref{SS}.
We would like some information on the exactness of these functors;
for this we need a hypothesis.

If $M$ is a profinite $\ZZ_3[[\GG_2]]$-module then 
$$
\Hom_{\Z_3[[\GG_2^1]]}(C_{\bullet},M) = 
\lim_{\a}\Hom_{\Z_3[[\GG_2^1]]}(C_{\bullet},M_\alpha)
$$
is also necessarily profinite as a $\ZZ_3$-module. Since profinite
$\ZZ_3$-modules are closed under kernels and cokernels,
the groups $E_2^{p,0}(M)$
are also profinite. We will use later that if $M$ is a profinite $\ZZ_3$-module
and $M/3M = 0$, then $M=0$.
\end{rem}

\begin{lem}\label{when-exact}  Suppose  $0\to M_1\to M_2\to M_3\to 0$
is an exact sequence of  profinite $\Z_3[[\GG_2^1]]$-modules such that
$H^1(G_{24},M_1)=0$.  Then there is a long exact sequence of
profinite $\ZZ_3$-modules
$$
0\to E_2^{0,0}(M_1)\to E_2^{0,0}(M_2)\to E_2^{0,0}(M_3)\to E_2^{1,0}(M_1)
\to \ldots \to E_2^{3,0}(M_2)\to E_2^{3,0}(M_3)\to 0 \ . 
$$ 
\end{lem}  

\begin{proof} In general the sequence of complexes 
$$
0\to \Hom_{\Z_3[[\GG_2^1]]}(C_{\bullet},M_1)
\to \Hom_{\Z_3[[\GG_2^1]]}(C_{\bullet},M_2)
\to \Hom_{\Z_3[[\GG_2^1]]}(C_{\bullet},M_3)\to 0 
$$ 
of profinite $\ZZ_3$-modules need not be exact; however, 
by Corollary \ref{SS}, the failure of exactness is exactly measured by
$H^1(G_{24},M_1)$. Therefore, if that group is zero,
then we do get an exact sequence of complexes, and the result
follows.
\end{proof}  

\begin{rem}\label{GHMRshort} 
By \cite{GHMR}, the resolution
$C_\bullet \to \ZZ_3$ of Theorem \ref{alg-res}  can be
promoted to a resolution 
of $\Et_\ast E_2^{h\GG_2^1}$ by twisted $\GG_2$-modules
\begin{equation}\label{e-res}
\begin{array}{ll}
\Et_\ast E_2^{h\GG_2^1} \to \Et_\ast E_2^{hG_{24}} \to&
\Et_\ast \Sigma^8E_2^{hSD_{16}}\\
\\
&\to \Et_\ast \Sigma^{8}E_2^{hSD_{16}} \to \Et_\ast
E_2^{hG_{24}} \to 0.
\end{array}
\end{equation}
We have $\Sigma^8E_2^{hSD_{16}}$ because $C_1$ is twisted
by a character. From  \S 5 of \cite{GHMR} we get the following
topological refinement: there is a sequence of maps between
spectra
\begin{align}\label{top-res}
E_2^{h\GG_2^1} \to E_2^{hG_{24}} \to
\Sigma^8E_2^{hSD_{16}}
& \to \Sigma^{40}E_2^{hSD_{16}} \to \Sigma^{48} E_2^{hG_{24}} 
\end{align}
realizing the resolution (\ref{e-res}) and
with the property that any two successive maps are null-homotopic
and all possible Toda brackets are zero modulo indeterminacy.
Note that there is an equivalence $\Sigma^8E_2^{hSD_{16}} \simeq 
\Sigma^{40}E_2^{hSD_{16}}$,
so that suspension is for symmetry only; however,
$$
\Sigma^{48} E_2^{hG_{24}} \not\simeq E_2^{hG_{24}}
$$
even though
$$
\Et_\ast \Sigma^{48} E_2^{hG_{24}} \cong \Et_\ast E_2^{hG_{24}}.
$$
This suspension is needed to make the Toda brackets vanish.
Because these Toda brackets vanish, the sequence of maps in the topological
complex (\ref{top-res}) further refines
to  a tower of fibrations
\begin{equation}\label{key-tower}
\xymatrix@R=15pt{
E_2^{h\GG_2^1} \rto & X_2\rto & X_1 \rto &E^{hG_{24}}\\
\Sigma^{45} E_2^{hG_{24}}\ar[u]&
\Sigma^{38}E_2^{hSD_{16}}\ar[u] & \Sigma^7E_2^{hSD_{16}}\ar[u]
}
\end{equation}
There is a similar tower for the sphere itself, using the resolution of
Remark \ref{full-grp}.
\end{rem}

\begin{rem}\label{alg-vs-top-ss} Let $\Sigma^{-p}F_p$ 
denote the successive
fibers in the tower (\ref{key-tower}); thus, for example,
$F_3 = \Sigma^{48}E_2^{hG_{24}}$. Then combining 
the descent spectral sequences for the groups $G_{24}$, 
$SD_{16}$ and $\GG_2^1$ with Corollary \ref{SS}
and the spectral sequence of the tower, we get a square of spectral
sequences 
\begin{equation}\label{diagram-of-ss}
\xymatrix{
E_1^{p,q}(\Et_tX)\ar@{=>}[d]_*+[o][F-]{1}  \ar@{=>}[r]^*+[o][F-]{2} &
H^{p+q}(\GG_2^1,\Et_tX) \ar@{=>}[d]^*+[o][F-]{3} \\
\pi_{t-q}L_{K(2)}(F_p\wedge X) \ar@{=>}[r]_*+[o][F-]{4}  &
\pi_{t-(p+q)} L_{K(2)}(E_2^{h\GG_2^1} \wedge X)\ .
}
\end{equation} 
We will use information about spectral sequences 
(1) and (2) to deduce information about spectral sequences (3) and (4). 
See Lemmas \ref{yet-to-be-pre} and \ref{yet-to-be-pre-k1}.

There is a similar square of spectral sequences where the lower right
corner becomes $\pi_\ast L_{K(2)}S^0$. This uses the resolution
of Remark \ref{full-grp} and the subsequent tower for the sphere.
\end{rem}

\section{The algebraic spectral sequences in the case of $(E_2)_*/(3)$}

Let $S/3$ denote the mod-$3$ Moore spectrum. 
Then, in the case of $(E_2)_*/(3)=(E_2)_\ast(S/3)$  
the spectral sequence  of Corollary \ref{SS} 
was completely worked out in \cite{HKM}.   We begin 
with some of the details.

First note that this is a spectral sequence  of modules over 
$H^*(\GG_2;(E_2)_*/(3))$. We will describe the
$E_1$-term as a module over the subalgebra
$$\FF_3[\b,v_1] \subseteq H^*(\GG_2;(E_2)_*/(3))
$$
where $\b\in H^2(\GG_2,(E_2)_{12}/(3))$ 
detects the image of the homotopy element 
$\b_1\in \pi_{10}S^0$ in $\pi_{10}(L_{K(2)}(S/3))$ and $v_1$
detects the image  of the homotopy element in $\pi_4(S/3)$
$$
\xymatrix{
S^4 \rto & \Sigma^4(S/3) \rto^-A &S/3
}
$$
of the inclusion of the bottom cell composed with the $v_1$-periodic map
constructed by Adams. 

In the next result, the element 
$\a$ of bidegree $(1,4)$ detects the image of the homotopy element 
$\a_1\in \pi_{3}S^0$  and the element 
$\widetilde{\a}$ of bidegree $(1,12)$ detects an element
in $\pi_{11}(L_{K(2)}(S/3))$ which maps to the image of $\beta_1$ in 
$\pi_{10}(L_{K(2)}S^0)$ under the 
pinch map $S/3 \to S^1$ to the top cell. For more 
details on these elements, as well as for the proof of the following theorem
we refer to \cite{HKM}.  We write 
$$
E_r^{p,q,t} = E_r^{p,q}((E_2)_t/3)
$$ 
for the $E_r$-term of the spectral sequence of Corollary \ref{SS}. For example, if
$p=0$ or $p=3$, then
$$
E_1^{p,\ast,t} = H^\ast(G_{24},(E_2)_t/3).
$$
By the calculations of  \cite{GHMR} \S 3, there is an invertible class $\Delta \in H^0(G_{24},
(E_2)_{24})$. We also write $\Delta$ for its image in $H^\ast(G_{24},(E_2)_t/3)$.
   
\begin{thm}\label{E1} There are isomorphisms of
$\FF_3[\beta,v_1]$-modules, with $\beta$ acting trivially on $E_1^{p,*,*}$ if $p=1,2$:
$$
E_1^{p,*,*}\cong 
\begin{cases}
\FF_3[[v_1^6\Delta^{-1}]][\Delta^{\pm 1},v_1,\beta,\alpha,\widetilde{\alpha}]
/(\alpha^2,\widetilde{\alpha}^2,v_1\alpha,v_1\widetilde{\alpha},
\alpha\widetilde{\alpha}+v_1\beta)e_p 
& p=0,3  \\
\\
\omega^2u^4\FF_3[[u_1^4]][u_1u^{-2},u^{\pm 8}]e_p  & p=1,2 \ .\\ 	
\end{cases}  
$$
\end{thm}

\begin{rem}\label{explain-notation} The module generators
$e_p$ are of tridegree $(p,0,0)$.  If $p=0$ or $p=3$, then $E_1^{p,0,*}$
is isomorphic to a completion of the ring
of mod-$3$  modular forms for smooth elliptic curves. 
Indeed, by Deligne's calculations \cite{Deligne} \S 6, the
ring of modular forms is $\FF_3[b_2,\Delta^{\pm 1}]$ where 
$b_2$ is the Hasse invariant and $\Delta$ is the discriminant. The Hasse
invariant of an elliptic curve can be computed as $v_1$ of the
associated formal group, so we can write $b_2=v_1$. 

If $p=1$ or $p=2$, we have written $E_1^{p,0,*}$ as 
a submodule of $(E_2)_\ast/(3) = \FF_9[[u_1]][u^{\pm 1}]$. 
Recall that there is a $3$-typical choice for the universal deformation
of the Honda formal group $\Gamma_2$ with
$v_1 = u_1u^{\pm 2}$ and $v_2 = u^{-8}$. 
\end{rem}

All differentials in the spectral sequence of Corollary \ref{SS}
with $M = (E_2)_\ast/(3)$ are $v_1$-linear. In particular, $d_1$ is 
determined by continuity and the following formulae 
established in \cite{HKM}. 

\begin{thm}\label{E1diff} 
There are elements 
$$
\begin{array}{llll}
\Delta_k\in E_1^{0,0,24k},&  b_{2k+1}\in E_1^{1,0,16k+8}, 
& \overline{b}_{2k+1}\in E_1^{2,0,16k+8}, 
& \overline{\Delta}_k\in E_1^{3,0,24k} 
\end{array}
$$ 
for each $k\in \Z$ satisfying   
$$
\begin{array}{llll} 
\Delta_k\equiv\Delta^ke_0, & b_{2k+1}\equiv \omega^2u^{-4(2k+1)}e_1, 
& \overline{b}_{2k+1}\equiv \omega^2u^{-4(2k+1)}e_2, 
& \overline{\Delta}_k\equiv\Delta^ke_3
\end{array}
$$  
(where the congruences are modulo the ideal 
$(v_1^6\Delta^{-1})$ resp. $(v_1^4u^8)$ and in the case of $\Delta_0$ we 
even 
have equality $\Delta_0=\Delta^0e_0=e_0$)
such that  
$$
\begin{array}{rcl}
d_1(\Delta_k)&=&  
\begin{cases}
(-1)^{m+1}b_{2.(3m+1)+1}                  &k=2m+1, m \in \ZZ \\
(-1)^{m+1}mv_1^{4.3^n-2}b_{2.3^n(3m-1)+1} &k=2m.3^n, m \in \ZZ, 
m\not\equiv 0\bmod(3), n\geq 0\\ 
0                                         &k=0  \\
\end{cases}   \\
&&\\
d_1(b_{2k+1})&=&  
\begin{cases}
(-1)^nv_1^{6.3^{n}+2}\overline{b}_{3^{n+1}(6m+1)}
&\ \ \ \ \ \ \ k=3^{n+1}(3m+1), m \in \ZZ,n\geq 0 \\
(-1)^nv_1^{10.3^n+2}\overline{b}_{3^n(18m+11)}   
&\ \ \ \ \ \ \ k=3^{n}(9m+8), m \in \ZZ,n\geq 0  \\ 
0                                                
&\ \ \ \ \ \ \ \textrm{else}\\
\end{cases}  \\
&&\\  
d_1(\overline{b}_{2k+1})&=&
\begin{cases}
(-1)^{m+1}v_1^{2}\overline{\Delta}_{2m}             & 2k+1=6m+1, m \in \ZZ \\
(-1)^{m+n}v_1^{4.3^n}\overline{\Delta}_{3^{n}(6m+5)}& 2k+1=3^{n}(18m+17),
 m \in \ZZ,n\geq 0 \\
(-1)^{m+n+1}v_1^{4.3^n}\overline{\Delta}_{3^{n}(6m+1)}&
2k+1=3^{n}(18m+5), m \in \ZZ,n\geq 0\\
0                                         & \textrm{else}  \ .
\end{cases} 
\end{array} 
$$ 
\end{thm}

We will actually only need the following consequence of these results,
which follows after a little bookkeeping.

\begin{cor}\label{whats-needed} There is an isomorphism
\begin{equation*}
E_2^{p,0}((E_2)_0/(3))\cong 
\begin{cases} 
\F_3 & p=0,3 \\ 
0    & p=1,2 \ . \\
\end{cases} 
\end{equation*}
\end{cor}

\begin{rem}\label{integral-calc}
We also record here the integral calculation $H^\ast(G_{24},(E_2)_\ast)$
from \cite{GHMR}; we will use this in Proposition \ref{k1-calcs}.
There are elements $c_4$, $c_6$ and $\Delta$ in
$H^0(G_{24},(E_2)_\ast)$ of internal degrees $8$, $12$ and $24$
respectively. The element $\Delta$ is invertible and there is a relation
$$
c_4^3 - c_6^2 = (12)^3\Delta\ .
$$
Define $j=c_4^3/\Delta$ and let $M_\ast$ be the graded ring
$$
M_\ast = \ZZ_3[[j]][c_4,c_6,\Delta^{\pm 1}]/(c_4^3 - c_6^2 = (12)^3\Delta,
\Delta j = c_4^3).
$$
There are also elements $\alpha \in H^1(G_{24},(E_2)_4)$ and
$\beta \in H^2(G_{24},(E_2)_{12})$ which reduce to the restriction 
(from $\GG_2^1$ to $G_{24}$) of the elements of the
same name in Theorem \ref{E1}. There are relations
\begin{eqnarray}\label{mod-rels}
3\alpha = 3\beta = \alpha^2 &=0\  \nonumber\\
c_4 \alpha = c_4\beta &=0\ \\
c_6 \alpha = c_6 \beta &=0. \nonumber
\end{eqnarray}
Finally
$$
H^\ast (G_{24},(E_2)_\ast ) = M_\ast[\alpha,\beta]/R
$$
where $R$ is the ideal of relations given by (\ref{mod-rels}). The element $\Delta$ has
already appeared in Theorem \ref{E1}. Modulo $3$, $c_4 \equiv v_1^2$ and $c_6 \equiv v_1^3$
up to a unit in $H^0(G_{24},(E_2)_0/3) = \FF_3[[j]]$. Compare \cite{GHMV1}, Proposition 7.
\end{rem}

\section{The rational calculation}

The purpose of this section is to give enough qualitative information
about the integral calculation of $H^\ast(\GG_2,(E_2)_\ast)$ 
in order to prove Theorem \ref{maintheorem}. Much of this is more
refined than we actually need, but of interest in its own 
right.

The following result implies that the rational homotopy will all arise
from $H^\ast(\GG_2,(E_2)_0)$.

\begin{prop}\label{rational}
a) Suppose $t=4.3^km$ with $m\not\equiv 0\ \text{mod}\ 
(3)$.   Then $3^{k+1}H^*(\GG_2,(E_2)_{t})=0$. 

b)
Suppose $t$ is not divisible by $4$. Then $H^\ast(\GG_2,(E_2)_t)=0$.
\end{prop}   

\begin{proof} Part (b) is the usual sparseness for the Adams-Novikov Spectral
Sequence. We can prove this here by considering the spectral sequence  
$$
H^p(\GG_2/\{\pm 1\},H^q(\{\pm 1\},(E_2)_t) \Longrightarrow
H^{p+q}(\GG_2,(E_2)_t)
$$
given by the inclusion of the central subgroup 
$\{\pm 1\}\subset\Z_3^{\times}\subset\GG_2$. 
The central subgroup $\Z_3^\times$  acts trivially
on $(E_2)_0$ and  by multiplication on
$u$;  that is, if $g\in \Z_3^{\times}$   then $g_*(u)=gu$. 
In particular  we find 
$$
H^q(\{\pm 1\},(E_2)_t)  = 0
$$
unless $t$ is a non-zero multiple of 4 and $q=0$. From this (b) follows.

For (a) we use the spectral sequence
$$
H^p(\GG_2^1,H^q(\ZZ_3,(E_2)_t)) \Longrightarrow
H^{p+q}(\GG_2,(E_2)_t)
$$
If $\psi \in \ZZ_3$ is a topological generator,
then $\psi\equiv 4$ modulo $9$. In particular, 
$$
\psi(u^{t/2}) = (1 + 2.3^{k+1}m)u^{t/2}\qquad \mathrm{mod}\ (3^{k+2})
$$
and we have that
$H^q(\ZZ_3,(E_2)_t) = 0$ unless $q=1$ and
$$
3^{k+1}H^1(\ZZ_3^\times,(E_2)_t)=0.
$$
Then (a) follows.
\end{proof}

It's not possible to be quite so precise in the case of $\GG_2^1$. However,
we do have the following result.

\begin{prop}\label{rational-bis} Suppose $s > 3$ or $t$ is not divisible by $4$.
Then
$$
H^s(\GG_2^1,E_t) \otimes \QQ = 0.
$$
\end{prop}

\begin{proof} This follows from tensoring the spectral sequence of 
Corollary \ref{SS} with $\QQ$ and noting that
$$
H^s(G_{24},E_t) \otimes \QQ = H^s(SD_{16},E_t) \otimes \QQ = 0
$$
if $s > 0$ or $t$ is not divisible by $4$.
\end{proof}

To isolate the torsion-free part of the cohomology of either $\GG_2$
or $\GG_2^1$ we use the
spectral sequences of Corollary \ref{SS}. 
From Remark \ref{integral-calc} we have an inclusion which is an isomorphism
in positive degrees
$$
\Z_3[\b^2\Delta^{-1}]/(3\b^2\Delta^{-1}) \subseteq H^\ast (G_{24},(E_2)_0).
$$
In the notation of the spectral sequences of Corollary \ref{SS} and 
Remark \ref{full-grp} we then have inclusions
$$
\Z_3[\b^2\Delta^{-1}]/(3\b^2\Delta^{-1})e_p \subseteq
E_1^{p,\ast}(\GG^1_2,(E_2)_0),\quad p=0,3. 
$$ 

Here is the main algebraic result. 

\begin{thm}\label{rational2} a) There is an element
$e\in H^3(\GG_2^1,(E_2)_0)$ of infinite order so that  
$$ 
H^\ast (\GG_2^1,(E_2)_0)\cong 
\Lambda(e)\otimes\Z_3[\b^2\Delta^{-1}]/(3\b^2\Delta^{-1})\ .
$$   

b) There is an element $\zeta\in H^1(\GG_2,(E_2)_0)$  of infinite order
so that  
$$ 
H^\ast (\GG_2,(E_2)_0)\cong 
\Lambda(\zeta)\otimes H^\ast (\GG_2^1,(E_2)_0) \ .
$$
\end{thm}

\begin{proof}  
For the proof of part (a) we consider the functors 
from the category of profinite $\Z_3[[\GG_2^1]]$-modules 
to $3$-profinite abelian groups introduced in Remark 
\ref{profinite} and given by 
$$
M\mapsto E_2^{p,0}(M) =H^p(\Hom_{\Z_3[[\GG_2^1]]}(C_{\bullet},M))\ . 
$$ 
Here $C_\bullet$ is the resolution of Theorem \ref{alg-res}.

From Remark \ref{integral-calc} we know 
that the hypothesis of Lemma \ref{when-exact}
is satisfied for the  short exact sequence 
$$
0\to (E_2)_0\buildrel{\times 3}\over\longrightarrow (E_2)_0
\to (E_2)_0/(3)\to 0. \  
$$ 
Then Corollary \ref{whats-needed}, 
the long exact sequence of Lemma \ref{when-exact}, and the fact 
that the groups $E_2^{p,0}(\GG_2^1,(E_2)_0)$ are profinite $\ZZ_3$-
modules give  
$$
E_2^{p,0}(\GG_2^1,(E_2)_0)\cong 
\begin{cases} 
\Z_3, & p=0,3; \\ 
0,    & p=1,2 \ . \\
\end{cases} 
$$ 
See Remark \ref{profinite}.
This implies that the $E_2$-term of the spectral sequence of Corollary 
\ref{SS}  is isomorphic to
$$
\Lambda(e_3) \otimes \ \Z_3[\b^2\Delta^{-1}]/(3\b^2\Delta^{-1}).
$$
Since there can be no further differentials, part (a) follows. 

Since the central $\ZZ_3$ acts trivially on $(E_2)_0$, we have a 
K\"unneth isomorphism
$$
H^\ast(\ZZ_3,\ZZ_3)  \otimes H^\ast (\GG_2^1,(E_2)_0) \cong
H^\ast(\GG_2,(E_2)_0).
$$
Part (b) follows.  
\end{proof} 

We are now ready to state and prove the main result on rational homotopy.
Note that Theorem \ref{maintheorem} 
of the introduction is an immediate consequence of
Proposition \ref{rational}, of Theorem \ref{rational2}, 
of the spectral sequence
$$
H^s(\GG_2,(E_2)_t) \otimes \QQ \Longrightarrow
\pi_{t-s}L_{K(2)}S^0 \otimes \QQ.
$$ 
and part (b) of the following Lemma.  

Let $\kappa_2$ be the set of isomorphism
classes of $K(2)$-local spectra $X$ so that $(E_2)_\ast X \cong (E_2)_\ast
= (E_2)_\ast S^0$ as twisted $\GG_2$-modules. This is 
a  subgroup of the $K(2)$-local Picard group; the group operation
is given by smash product. In \cite{GHMRpic} we show that
$\kappa_2 \cong (\ZZ/3)^2$. 

For the next result, the spectra $F_p$ were defined in Remark
\ref{alg-vs-top-ss}.
 
\begin{lem}\label{yet-to-be-pre} (a)  Let $X \in \kappa_2$. 
Then for $p=0,1,2,3$, the edge homomorphism of the localized 
descent spectral sequence 
$$
E_2^{p,q,t}=\Ext_{\ZZ_3[[\GG_2^1]]}^q(C_p,(E_tX))\otimes \QQ
\Longrightarrow \pi_{t-q}L_{K(2)}(F_p\wedge X)  \otimes \QQ
$$
induces an isomorphism
$$
\xymatrix{
\pi_\ast L_{K(2)}(F_p\wedge X)\otimes \QQ \cong 
\Hom_{\ZZ_3[[\GG_2^1]]}(C_p,(E_2)_\ast X) \otimes \QQ.
}
$$
(b) Let $F = \GG_2^1$ or $\GG_2$. Then the localized spectral
sequence
$$
H^s(F,(E_tX)) \otimes \QQ \Longrightarrow \pi_{t-s}L_{K(2)}(E_2^{hF} \wedge
X) \otimes \QQ
$$
converges and collapses.
\end{lem}

\begin{proof} For (a), the spectral sequence
$$
H^s(F,(E_2)_tX) \Longrightarrow \pi_{t-s}L_{K(2)}(E_2^{hF}\wedge X)
$$
has a horizontal vanishing line at $E_\infty$ by the calculations
of \S 3 of \cite{GHMR}.
Thus the rationalized spectral sequence 
$$
H^s(F,(E_2)_tX)\otimes \QQ \Longrightarrow
\pi_{t-s}L_{K(2)}(E_2^{hF}\wedge X) \otimes \QQ
$$
converges. The result follows in this case.

For (b) we first do the case of $\GG_2^1$. We localize the square of spectral 
sequences of (\ref{diagram-of-ss}) to get a new square of 
spectral sequences
\begin{equation}\label{diagram-of-ss-rat}
\xymatrix{
E_1^{p,q}(\Et_tX)\otimes \QQ\ar@{=>}[d]_*+[o][F-]{1}
 \ar@{=>}[r]^*+[o][F-]{2} &
H^{p+q}(\GG_2^1,\Et_t X) \otimes \QQ \ar@{=>}[d]^*+[o][F-]{3} \\
\pi_{t-q}L_{K(2)}(F_p\wedge X) \otimes \QQ 
\ar@{=>}[r]_*+[o][F-]{4}  & \pi_{t-(p+q)} 
L_{K(2)}(E_2^{h\GG_2^1} \wedge X) \otimes \QQ.
}
\end{equation} 
We will show that spectral sequence (3) converges and the result will
follow.

First note that spectral sequences (2) and (4) are the localizations of
finite and convergent spectral sequences, so must converge. We have
noted in the proof of part (a) that the spectral sequences of (1) converge. 
Now we note that spectral sequence (2) with $q=0$ and 
the spectral sequence of (4) have the same $d_1$, by the construction
of the tower.

From this we conclude the $E_2$-term of the spectral sequence (4) is
$$
E_2^{p,t} \cong H^p(\GG_2^1,(E_2)_tX) \otimes \QQ.
$$
Proposition \ref{rational-bis} implies that the spectral sequence (4) collapses
and that, in fact, if 
$$
\pi_nL_{K(2)}(E_2^{h\GG_2^1} \wedge X) \otimes \QQ \ne 0
$$
there are unique integers $p$ and $t$ with $t-p=n$ and
$$
\pi_nL_{K(2)}(E_2^{h\GG_2^1} \wedge X) \otimes \QQ \cong
H^p(\GG_2^1,(E_2)_tX) \otimes \QQ.
$$
It follows immediately that spectral sequence (3) converges and collapses.

There is an analogous argument for $\GG_2$, using the expanded
square of spectral sequences for this group. See Remark
\ref{alg-vs-top-ss}. The needed properties of $H^p(\GG_2,(E_2)_tX) \otimes \QQ$
are obtained by combining Proposition \ref{rational} with Theorem \ref{rational2}.b.
\end{proof}

Theorem \ref{rational2} and Lemma \ref{yet-to-be-pre}
immediately imply the following results. Let $S^n_p$ denote
the $p$-complete sphere.

\begin{thm}\label{yet-to-be} Let $X \in \kappa_2$. Then the 
rational Hurewicz homomorphism
$$
\pi_0L_0X \longr \pi_0L_0L_{K(2)}(E_2 \wedge X)\cong (E_2)_0X \otimes \QQ
$$ 
is injective. Given a choice of isomorphism $f:(E_2)_\ast \to (E_2)_\ast X$
of twisted $\GG_2$-modules the image of the multiplicative unit $1$ under the isomorphism
$$
\QQ_3 \cong \QQ \otimes H^0(\GG_2,(E_2)_0) \cong \pi_0L_0X
$$
extends to a weak equivalence of $L_0L_{K(2)}S^0$-modules
$$
L_0L_{K(2)}S^0 \simeq L_0X.
$$
\end{thm}
\bigskip

\begin{thm}\label{ratsphere} The localized spectral sequence of Lemma \ref{yet-to-be-pre}
$$
\QQ \otimes H^s(\GG_2,(E_2)_t) \Longrightarrow \pi_{s-t}L_0L_{K(2)}S^0
$$
determines an isomorphism
$$
\Lambda_{\QQ_3}(\zeta,e) \cong \pi_\ast L_0L_{K(2)}S^0.
$$
Furthermore, there is a weak equivalence
\begin{equation*}\label{L0}
L_0(S_3^0\vee S_3^{-1}\vee S_3^{-3}\vee
S_3^{-4})\simeq L_0L_{K(2)}S^0.
\end{equation*}
\end{thm}

\section{The chromatic splitting conjecture} 

In this section we prove a refinement of Theorem \ref{maintheorem2} of the
introduction. 

Our main result, Theorem \ref{chrom-split-1}, analyzes $L_1X$ for
$X \in \kappa_2$.
For this we will use the chromatic fracture square
\begin{equation}\label{chrm-sq}
\xymatrix{
L_1X \rto \dto & L_{K(1)}X \dto\\
L_0X \rto & L_0L_{K(1)}X
}
\end{equation}
We made an analysis of $L_0X$ in Theorem \ref{yet-to-be}. The
calculation of  $L_{K(1)}X$ has a number of interesting features, so
we dwell on it. In particular, we will produce weak equivalences
$$
L_{K(1)}S^0 \to L_{K(1)}L_{K(2)}(E_2^{h\GG_2^1}\wedge X)
$$
which will be the key to the entire calculation.

We begin with the following general result; we learned the argument
from Mark Hovey. The argument is valid only for ${K(1)}$-localization,
which may indicate that it would be hard to generalize our arguments to
higher height. Let $S/p^n$ denote the Moore spectrum.  

\begin{lem}\label{k1-localize} Let $X$ be a spectrum. Then
$$
L_{K(1)}X = \holim_n\ v_1^{-1} S/p^n\wedge X
$$
and $v_1^t:\Sigma^{2t(p-1)} S/p^n \to S/p^n$ is any choice
of $v_1$-self map. 
\end{lem}

\begin{proof} By Proposition 7.10(e) of
\cite{HS} we know that
\begin{equation}\label{HScor}
L_{K(1)}X = \holim_n\  S/p^n\wedge L_1 X.
\end{equation}
Since $L_1$ is smashing, so we may rewrite (\ref{HScor}) as 
$$
L_{K(1)}X = \holim_n\ L_1 (S/p^n) \wedge X.
$$
Thus it is sufficient to know $L_1 (S/p^n) = v_1^{-1}S/p^n$.  This follows
from the calculations of  \cite{MillerJPAA}; see \cite{RavTel} for complete details. 
\end{proof}

\def\mm{{{\mathfrak{m}}}}

If $R$ is a discrete ring, then the Laurent series over $R$ 
is the ring $R((x)) = R[[x]][x^{-1}]$.  

\begin{prop}\label{local-e2s} (a) There are isomorphisms
\begin{equation}\label{the-local-calc-G24}
\FF_3((v_1^6\Delta^{-1}))[v_1^{\pm 1}] \cong v_1^{-1}H_\ast(G_{24},(E_2)_\ast/3)
\end{equation}
and
\begin{equation}\label{the-local-calc-SD16}
\FF_3((v_1^4v_2^{-1}))[v_1^{\pm 1}] \cong
v_1^{-1} H^\ast (SD_{16},(E_2)_\ast/3).
\end{equation}

(b) There are isomorphisms
\begin{equation}\label{the-local-calc}
\FF_3[v_1^{\pm1}] \otimes \Lambda(v_1^{-1}b_1) \cong
 v_1^{-1} H^\ast (\GG^1_2,(E_2)_\ast/3).
\end{equation}
and
\begin{equation}\label{the-local-calc-big}
\FF_3[v_1^{\pm1}] \otimes \Lambda(v_1^{-1}b_1,\zeta) \cong
 v_1^{-1} H^\ast (\GG_2,(E_2)_\ast/3).
\end{equation}
The element $b_1$ has bidegree $(1,8)$ and the element
$v^{-1}b_1$ detects the image of the homotopy class $\alpha_1 \in
\pi_3 S^0/3$.  The element  $\zeta$ has bidegree $(1,0)$ and
is the image of the class of the same name in $H^1(\GG_2,(E_2)_0)$
from Theorem \ref{rational2}.b.
\end{prop}

\begin{proof} The results in (a) are immediate consequences of Theorem \ref{E1}. 
See also \cite{GHMR} \S 3.
For (b), the two isomorphisms both follow from Theorem
\ref{E1diff} and the algebraic  spectral sequences of Corollary \ref{SS}.
That $v_1^{-1}b_1$ detects the image of $\alpha_1$ is proved in
Proposition 1.5 of \cite{HKM}.
\end{proof}
 
Here is our key lemma. Compare Lemma \ref{yet-to-be-pre} in the rational
case.

\begin{lem}\label{yet-to-be-pre-k1} Let $X\in \kappa_2$ and let $X/3 = S/3\wedge X$.

(a) Suppose that
$F = G_{24}$ or $SD_{16}$. Then the edge homomorphism
induces an isomorphism
$$
\xymatrix{
\pi_\ast L_{K(1)}L_{K(2)}(E_2^{hF}\wedge X/3) \rto^-\cong &
v_1^{-1} H^0(F,(E_2)_\ast X/3)
}
$$

(b) Let $F = \GG_2^1$ or $\GG_2$. Then the localized spectral
sequence
$$
(v_1^{-1} H^s(F,(E_\ast X)/3))_t  \Longrightarrow \pi_{t-s}L_{K(1)}
(E_2^{hF} \wedge X/3)
$$
converges and collapses.
\end{lem}

\begin{proof} The proof of Lemma \ref{yet-to-be-pre} goes through
{\it mutatis mutandis}. We need only replace the localization
$H^\ast(F,M) \mapsto H^\ast(F,M) \otimes \QQ$ with the localization
$$
H^\ast(F,M) \longmapsto v_1^{-1}H^\ast(F,M/3)
$$
throughout, and use Theorem \ref{E1diff} in place of
Proposition \ref{rational} and Theorem \ref{rational2}.
\end{proof}

We now have the following remarkable calculation.

\begin{prop}\label{L1v0} Let $X\in \kappa_2$.
Then the $K(1)$-localized Hurewicz homomorphism  
$$
\pi_0 L_{K(1)}X/3 \longr \pi_0L_{K(1)}L_{K(2)}(E_2 \wedge X/3)
$$
is injective. Any choice of isomorphism $\Et_\ast \cong \Et_\ast X$ of
twisted $\GG_2$-modules uniquely defines a generator of
$$
\pi_0(L_{K(1)}L_{K(2)}(E_2^{h\GG_2^1}\wedge X/3)) \cong (v_1^{-1}
H^0(\GG^1_2,(E_2)_*/3)_0 \cong \FF_3.
$$
This generator extends uniquely to a weak equivalence
$$
L_{K(1)}S^0/3 \simeq L_{K(1)}L_{K(2)}(E_2^{h\GG_2^1}\wedge X/3).
$$
\end{prop}

\begin{proof}  We use the localized spectral sequence 
$$
(v_1^{-1}H^\ast(\GG_2^1,(E_2)_*/3))_t \Longrightarrow 
\pi_{t-s}L_{K(1)}L_{K(2)}(E_2^{h\GG_2^1}\wedge X/3).
$$
This converges by Lemma \ref{yet-to-be-pre-k1}.b. The choice of 
isomorphism $\Et_\ast \cong \Et_\ast X$ is used to identify the
$E_2$-term. By the isomorphism of (\ref{the-local-calc}) this spectral
sequence collapses. By \cite{MillerJPAA}, we know that there is an
isomorphism
$$
\FF_3[v_1^{\pm1}] \otimes \Lambda(\alpha) \cong
\pi_\ast L_{K(1)}S/3.
$$
where $\alpha$ is the image of $\alpha_1 \in \pi_3S^0/3$.
The result now follows from Proposition \ref{local-e2s}.
\end{proof}

This result will be extended to an integral calculation in Proposition
\ref{L1v0int}.

For a complete local
ring $A$ with maximal ideal $\mm$ define
$$
A((x)) = \lim_k\ \Big\{A/\mm^k((x))\Big\}.
$$   
This a completion of the ring of Laurent series. 
Recall that $v_1 = u_1u^{-2}$ and $v_2=u^{-4}$
for the standard $p$-typical deformation of the Honda
formal group over $\Et_\ast$. As a first example, Lemma \ref{k1-localize} immediately
gives
\begin{equation}\label{k1e2}
\pi_\ast L_{K(1)}E_2 = \W((u_1))[u^{\pm 1}].
\end{equation}
We now give a calculation of $\pi_\ast L_{K(1)}E_2^{hF}$ for our two important
finite subgroups. The 
elements  $c_4$, $c_6$, $\Delta$
were all introduced in Remark \ref{integral-calc}.

\begin{prop}\label{k1-calcs} Let $X \in \kappa_2$ and fix an isomorphism
$(E_2)_\ast X \cong (E_2)_\ast$ of twisted $\GG_2$-modules.

(a) The edge homomorphism of the homotopy fixed point
spectral sequence induces an isomorphism
\begin{equation*}
\pi_\ast L_{K(1)} L_{K(2)} (E_2^{hG_{24}} \wedge X)\cong
\lim\ v_{1}^{-1} H^0(G_{24},(E_2)_\ast/3^n).
\end{equation*}
Define $b_2 = c_6/c_4$ and $j = c_4^3/\Delta$. Then these choices define an isomorphism
$$
\ZZ_3((j))[b_2^{\pm 1}] \cong \lim\ v_{1}^{-1} H^0(G_{24},(E_2)_\ast/3^n).
$$

(b) The edge homomorphism of the homotopy fixed point
spectral sequence induces an isomorphism
\begin{equation*}
\pi_\ast L_{K(1)} L_{K(2)}(E_2^{hSD_{16}}\wedge X) \cong
\lim\ v_{1}^{-1} H^0(SD_{16},(E_2)_\ast/3^n).
\end{equation*}
The element $v_1 =u_1u^{-2} \in (E_2)_\ast$ is invariant under the action
of $SD_{16}$ and gives an isomorphism
\begin{equation*}
\ZZ_3((w))[v_1^{\pm 1}] \cong \lim\ v_{1}^{-1} H^0(SD_{16},(E_2)_\ast/3^n)
\end{equation*}
where $w = v_1^4/v_2$.
\end{prop}

\begin{proof} For (a), the first isomorphism follows from Proposition \ref{local-e2s},
Lemma \ref{k1-localize}, and a five lemma argument. For
the second isomorphism, we know by Remark \ref{integral-calc}
that $c_4 \equiv v_1^2$ and $c_6\equiv  v_1^3$ modulo $3$ and up to a unit.
It follows that $c_4$ in the inverse limit and that we have a map
$$
\ZZ_3((j))[b_2^{\pm 1}] \to \lim\ v_{1}^{-1} H^0(G_{24},(E_2)_\ast/3^n).
$$
By Proposition \ref{local-e2s}, this map induces an isomorphism modulo $3$ and the
result follows.

Part (b)  follows directly from \cite{GHMR} \S 3.
\end{proof} 

\begin{lem}\label{kappa-and finite}  Let $X\in \kappa_2$ and
let $F=G_{24}$ or $SD_{16}$.
Given a choice of isomorphism $\Et_\ast \cong \Et_\ast X$ of
twisted $\GG_2$-modules the image of the multiplicative unit $1$
under the isomorphisms
$$
\lim (v^{-1}H^0(F,E_\ast/3^n))_0 \cong \lim (v^{-1}H^0(F,E_\ast X/3^n))_0 \cong
\pi_0L_{K(1)}L_{K(2)}(E_2^{hF} \wedge X)
$$
extends to a weak equivalence of
$L_{K(1)}E_2^{hF}$-modules
$$
L_{K(1)}E_2^{hF} \simeq L_{K(1)}L_{K(2)}(E_2^{hF} \wedge X).
$$
\end{lem}

\begin{proof} Let $Z = L_{K(2)}(E_2^{hF}\wedge X)$. 
By Proposition \ref{k1-calcs}, the given isomorphism of Morava modules determines
a map $g:S^0 \to L_{K(1)} Z$.  
By induction and a five lemma argument, the induced map
$S^0 \to   L_{K(1)} Z \wedge S/3^n$
extends to a weak equivalence of $L_{K(1)}E_2^{hF}$-modules
$$
L_{K(1)}E_2^{hF}/3^n \simeq L_{K(1)} Z \wedge S/3^n
$$
and the result follows from Proposition \ref{k1-localize}.
\end{proof}

\begin{thm}\label{L1v0int} Let $X\in \kappa_2$.
Then the localized Hurewicz homomorphism 
$$
\pi_0 L_{K(1)}X \longr \pi_0L_{K(1)}L_{K(2)}(E_2 \wedge X)
$$
is injective. Given a choice of isomorphism $\Et_\ast \cong \Et_\ast X$ of
twisted $\GG_2$-modules the image of the multiplicative unit $1$
under the isomorphisms
$$
\ZZ_3 \cong \lim(v_1^{-1} H^0(\GG^1_2,(E_2)_*/3^n))_0 \cong
\pi_0(L_{K(1)}L_{K(2)}(E_2^{h\GG_2^1}\wedge X))
$$
extends to a weak equivalence of
$L_{K(1)}S^0$-modules
$$
L_{K(1)}S^0 \simeq L_{K(1)}L_{K(2)}(E_2^{h\GG_2^1}\wedge X).
$$
\end{thm}

\begin{proof}  Let $Y = L_{K(2)}(E_2^{h\GG_2^1}\wedge X)$.
Take the tower of \ref{GHMRshort} and apply the localization functor
$L_{K(1)}L_{K(2)}(-\wedge X)$ to produce a tower with homotopy inverse 
limit $L_{K(1)}Y$.  By Lemma \ref{kappa-and finite}, the fibers are all
of the form $\Sigma^{8k} L_{K(1)}E_2^{hF}$ with $F=G_{24}$
or $F=SD_{16}$. Using Proposition \ref{k1-calcs}, we then see that the
map
$$
S^0 \to L_{K(1)}L_{K(2)}(E_2^{hG_{24}} \wedge X) \simeq
L_{K(1)}E_2^{hG_{24}}
$$
induced by  the given isomorphism of Morava modules 
lifts uniquely to a map 
$$
\iota: L_{K(1)}S^0 \to L_{K(1)}Y.
$$
By Proposition \ref{L1v0} this induces a weak equivalence
$$
L_{K(1)} S/3 \simeq L_{K(1)}Y \wedge S/3.
$$
Then, using the natural fiber sequence
$$
L_{K(1)} S/3 \wedge Y \to L_{K(1)}S/3^n \wedge Y \to 
L_{K(1)} S/3^{n-1} \wedge Y,
$$
induction, and Lemma \ref{k1-localize}, we obtain
the desired weak equivalence.
\end{proof}

The following is an immediate consequence of the Theorem \ref{L1v0int} which we record for
later use.

\begin{cor}\label{mod-ver-int} Let Let $X\in \kappa_2$.
Given a choice of isomorphism $\Et_\ast \cong \Et_\ast X$ of
twisted $\GG_2$-modules the image of the multiplicative unit $1$
under the isomorphisms
$$
\ZZ_3 \cong \lim(v_1^{-1} H^0(\GG^1_2,(E_2)_*/3^n))_0 \cong
\pi_0(L_{K(1)}L_{K(2)}(E_2^{h\GG_2^1}\wedge X))
$$
extends to a weak equivalence of
$L_{K(1)}E_2^{h\GG_2^1}$-modules
$$
L_{K(1)}E_2^{h\GG_2^1} \simeq L_{K(1)}L_{K(2)}(E_2^{h\GG_2^1}\wedge X).
$$
\end{cor}

We now want to extend Theorem \ref{L1v0int}  to the
sphere itself. Recall that there is a fiber sequence
\begin{equation*}
\xymatrix{
L_{K(2)}S^0 \rto & E_2^{h\GG_2^1} \rto^{\psi-1} &
E_2^{h\GG_2^1} }
\end{equation*}
where $\psi$ is a topological generator of the central $\ZZ_3 \subseteq
\GG_2$. For any $K(2)$-local $X$, 
we may apply the functor $L_{K(2)}((-)\wedge X)$ to get a fiber sequence
\begin{equation}\label{keyfiber}
\xymatrix{
X \rto & L_{K(2)}(E_2^{h\GG_2^1}\wedge X) \rto^{\psi-1} &
L_{K(2)}(E_2^{h\GG_2^1}\wedge X).
 }
\end{equation}

\begin{thm}\label{L1}  a) Let $X \in \kappa_2$. Given a choice of isomorphism
$\Et_\ast \cong \Et_\ast X$ of twisted $\GG_2$-modules the image of the multiplicative unit $1$
under the isomorphisms
$$
\ZZ_3 \cong \lim(v_1^{-1}H^0(\GG_2,(E_2)_*/3^n)_*)_0 \cong \pi_0L_{K(1)}X
$$
extends to a weak equivalence of $L_{K(1)}L_{K(2)}S^0$-modules  
$$
L_{K(1)}L_{K(2)}S^0 \simeq L_{K(1)}X.
$$

b) The weak equivalence $L_{K(1)}S^0 \simeq
L_{K(1)}L_{K(2)}E_2^{h\GG_2^1}$ of
Proposition \ref{L1v0int} factors
uniquely though $L_{K(1)}L_{K(2)}S^0$
and extends to a weak equivalence
$$
L_{K(1)}S^0 \vee L_{K(1)}S^{-1} \simeq L_{K(1)}L_{K(2)}S^0
$$
where $L_{K(1)}S^{-1} \to L_{K(1)}L_{K(2)}S^0$ is induced by $\zeta \in \pi_{-1}L_{K(2)}S^0$.
\end{thm}
 
\begin{proof} Let $f:L_{K(1)}E_2^{h\GG_2^1} \to L_{K(1)}L_{K(2)}(E_2^{h\GG_2^1} \wedge X)$
be the equivalence of Corollary \ref{mod-ver-int}.
Since $\psi:E_2^{h\GG_2^1} \to E_2^{h\GG_2^1}$
is a morphism of ring spectra, we get a diagram of $L_{K(1)}E_2^{h\GG_2^1}$-module
maps
$$
\xymatrix{
L_{K(1)}E_2^{h\GG_2^1} \rto^-f \dto_{\psi - 1} &
L_{K(1)}L_{K(2)}(E_2^{h\GG_2^1} \wedge X) \dto^{(\psi-1) \wedge X}\\
L_{K(1)}E_2^{h\GG_2^1} \rto^-f &
L_{K(1)}L_{K(2)}(E_2^{h\GG_2^1} \wedge X).
}
$$
By Theorem \ref{L1v0int}, there is an equivalence
$L_{K(1)}S^0 \simeq L_{K(1)}E_2^{h\GG_2^1}$.
Hence, to check that the diagram commutes, we
need only verify that if commutes after applying
$\pi_0$, and this is obvious. Part (a) follows.

We now prove part (b). Let 
$
f_0:L_{K(1)}S^0 \longr L_{K(2)}E_2^{h\GG_2^1}
$
be the equivalence, as in Theorem \ref{L1v0int}. The composition $(\psi-1)f_0$ is zero, 
as $\psi$ induces a ring map on $(E_2)_0$. Because
$\pi_1L_{K(1)}S^0 = 0$, $f_0$ lifts uniquely to a map $f:L_{K(1)}S^0 \to
L_{K(1)}L_{K(2)}S^0$ and we get a weak equivalence
$$
f \vee g: L_{K(1)}S^0 \vee L_{K(1)}S^{-1} \longr L_{K(1)}L_{K(2)}S^0
$$
where $g$ is the desuspension of the composition
$$
\xymatrix{
L_{K(1)}S^0 \rto^-{f_0}_-{\simeq} &L_{K(1)}E_2^{h\GG_2^1} \rto
& \Sigma L_{K(1)}L_{K(2)}S^0.
}
$$
As $\zeta$ is defined to be the image of unit in $\pi_0E_2^{h\GG_2^1}$ in
$\pi_{-1}L_{K(2)}S^0$, the result follows.
\end{proof}

We now come to our main theorems. 

\begin{thm}\label{chrom-split-1} Let $X \in \kappa_2$. Then the
localized Hurewicz homomorphism  
$$
\pi_0L_1X \longr \pi_0L_1L_{K(2)}(E_2 \wedge X)
$$
is injective. A choice
of isomorphism $f:(E_2)_\ast \to (E_2)_\ast X$  determines a generator 
of $\pi_0L_1X \cong \ZZ_3$. This generator extends
uniquely to a weak equivalence of $L_1L_{K(2)}S^0$-modules
$$
L_1L_{K(2)}S^0 \simeq L_1X.
$$
\end{thm}

\begin{proof} From
Theorem \ref{L1} we have that $\pi_1L_{K(1)}X=0$ for
all $X \in \kappa_2$. The result then follows by the
chromatic fracture square (\ref{chrm-sq}), Theorem \ref{yet-to-be} and
Theorem \ref{L1}.
\end{proof}

\begin{thm}[{\bf Chromatic Splitting}]\label{chrom-split-2} If $n=2$
and $p=3$, then
$$
L_1L_{K(2)}S^0\simeq L_1(S_3^0\vee S_3^{-1})\vee L_0(S_3^{-3}\vee 
S_3^{-4}).
$$
where $S^n_p$ denotes the $p$-complete sphere.
\end{thm}
  
\begin{proof}
We use the chromatic square of  (\ref{chrm-sq}). Let $X = L_{K(2)}S^0$.
Theorem \ref{L1} implies
$$
L_0L_{K(1)}X \simeq L_0L_{K(1)}(S^0 \vee S^{-1}).
$$
From Theorem \ref{yet-to-be} we have that
$$
L_0X \simeq L_0(S_3^0 \vee S_3^{-1} \vee S_3^{-3} \vee S_3^{-4}).
$$
Thus we need only show that the map
$$
L_0X \longrightarrow L_0L_{K(1)}X
$$
is equivalent to the composition
$$
L_0(S_3^0 \vee S_3^{-1} \vee S_3^{-3} \vee S_3^{-4})
\longr L_0(S_3^0 \vee S_3^{-1})\longr
L_0L_{K(1)}(S^0 \vee S^{-1})
$$
where the first map is projection and the second map is the $L_0$ localization
of the canonical map $S_3^0 \vee S_3^{-1} \to L_{K(1)}(S^0 \vee S^{-1})$.
This follows from Theorem \ref{ratsphere} and Theorem \ref{L1}.b.
\end{proof} 

\bibliographystyle{amsplain}
\bigbreak

\end{document}